\documentclass[12pt,reqno]{amsart}
\usepackage{amssymb,amsmath}\newcommand{\be}{\begin{eqnarray}}
\newcommand{\ee}{\end{eqnarray}}

\newcommand{\R}{{\mathbb R}}

\newcommand{\B}{{\mathcal B}}
\newcommand{\E}{{\bf E\,}}

\newcommand{\D}{{\mathcal D}}

\newcommand{\LLL}{\mathcal{L}}

\newtheorem{theorem}{Theorem}
\newtheorem{lemma}[theorem]{Lemma}
\newtheorem{cor}[theorem]{Corollary}

\theoremstyle{definition}
\newtheorem{defi}[theorem]{Definition}

\theoremstyle{remark}
\newtheorem{remark}[theorem]{Remark}

\numberwithin{equation}{section}\input epsf.sty

\begin{document}\thispagestyle{empty}

\title[Orthogonal martingales]{On Burkholder function for orthogonal  martingales  and zeros of Legendre polynomials}
\author[A.Borichev, P.Janakiraman, A.Volberg]{Alexander Borichev, Prabhu Janakiraman, Alexander Volberg}
\address{Alexander Borichev, 
\newline
Universit\'e de Provence, Marseille,
\newline {\tt borichev@cmi.univ-mrs.fr},
\newline \phantom{x}\,\, Prabhu Janakiraman, 
\newline Department of Mathematics, Michigan State University,
\newline {\tt pjanakir1978@gmail.com},
\newline \phantom{x}\,\, Alexander Volberg, 
\newline Department of Mathematics, Michigan State University,
\newline {\tt volberg@math.msu.edu}}

\thanks{The research of the first author was partially supported by the ANR grants DYNOP and FRAB; 
the research of the second and the third authors was partially supported by the NSF grants DMS-0501067 and DMS-0605166.
}

\begin{abstract}
Burkholder obtained a sharp estimate of $\E|W|^p$ via $\E|Z|^p$, for martingales $W$  differentially subordinated to martingales $Z$. His result is that 
$\E|W|^p\le (p^*-1)^p\E|Z|^p$, where $p^* =\max (p, \frac{p}{p-1})$. What happens if the martingales have 
an extra property of being orthogonal martingales? This property is an analog (for martingales) of 
the Cauchy-Riemann equation for functions, and it naturally appears in a problem on singular integrals
(see the references at the end of Section~1). 
We establish here that in this case the constant is quite different. Actually, $\E|W|^p\le (\frac{1+z_p}{1-z_p})^p\E|Z|^p$, 
$p\ge 2$, where $z_p$ is a specific zero of a certain solution of the Legendre ODE. We also prove the sharpness 
of this estimate. Asymptotically, $(1+z_p)/(1-z_p)=(4j^{-2}_0+o(1))p$, 
$p\to\infty$, where $j_0$ is the first positive zero of the Bessel function of zero order.
This connection with zeros of special functions (and orthogonal polynomials for $p=n(n+1)$) is rather unexpected.
\end{abstract}
\maketitle

\section{Introduction}

Let $Z=(X,Y)$, $W=(U,V)$ be two $\R^2$--valued martingales on the filtration of the $2$--dimensional Brownian motion 
$B_s= (B_{1s}, B_{2s})^T$. Let 
$$
A=\begin{bmatrix} -1, & i\\i, & 1\end{bmatrix}.
$$ 
We want $W$ to be a martingale transform of $Z$ defined by $A$. Let
\begin{align*}
X(t)&= \int_0^t \overrightarrow{x}(s)\cdot dB_s\,,\\
Y(t)&= \int_0^t\overrightarrow{y}(s)\cdot dB_s\,,
\end{align*}
where $X, Y$ are {\it real-valued} processes, and $\overrightarrow{x}(s), \overrightarrow{y}(s)$ are $\R^2$--valued ``martingale differences" written as row vectors. 
 
Put
\begin{equation*}
Z(t) = X(t) +iY(t),\quad Z(t) = \int_0^t (\overrightarrow{x}(s) +i \overrightarrow{y}(s))\cdot dB_s\,,
\end{equation*}
and
\begin{equation*}
W(t) = U(t) +iV(t),\quad W(t) = \int_0^t (A(\overrightarrow{x}(s)^T +i \overrightarrow{y}(s)^T))^T\cdot dB_s\,.
\end{equation*}
We denote
$$
W=A\star Z\,.
$$
As above, 
\begin{gather*}
U(t) = \int_0^t \overrightarrow{u}(s)\cdot dB_s\,,\\
V(t) = \int_0^t\overrightarrow{v}(s)\cdot dB_s\,,\\
W(t) = \int_0^t (\overrightarrow{u}(s) +i \overrightarrow{v}(s))\cdot dB_s\,.
\end{gather*}
 
We can easily write the components of $\overrightarrow{u}(s), \overrightarrow{v}(s)$:
\begin{gather*}
u_1(s) =-x_1(s)-y_2(s),\quad v_1(s) = x_2(s)-y_1(s)\,,\\
u_2(s) =x_2(s)-y_1(s),\quad v_2(s) =x_1(s)+y_2(s)\,.
\end{gather*}

Note that
\begin{equation*}
\overrightarrow{u}\cdot \overrightarrow{v} = u_1v_1+u_2v_2 = -(x_1+y_2)(x_2-y_1) + (x_2-y_1)(x_1+y_2)=0\,.
\end{equation*}

\subsection{Local orthogonality.}
\label{localort}

The processes
\begin{gather*}
\langle X, U\rangle(t) :=\int_0^t \overrightarrow{x}\cdot \overrightarrow{u}ds,\quad
\langle X, V\rangle(t) :=\int_0^t \overrightarrow{x}\cdot \overrightarrow{v}ds,\\
\langle Y, U\rangle(t) :=\int_0^t \overrightarrow{y}\cdot \overrightarrow{u}ds,\quad 
\langle Y, V\rangle(t) :=\int_0^t \overrightarrow{y}\cdot \overrightarrow{v}ds,\\
\langle X, X\rangle(t) :=\int_0^t \overrightarrow{x}\cdot \overrightarrow{x}ds,\quad 
\langle Y, Y\rangle(t) :=\int_0^t \overrightarrow{y}\cdot \overrightarrow{y}ds,\\
\langle X, Y\rangle(t) :=\int_0^t \overrightarrow{x}\cdot \overrightarrow{y}ds,\quad 
\langle U, U\rangle(t) :=\int_0^t \overrightarrow{u}\cdot \overrightarrow{u}ds,\\
\langle V, V\rangle(t) :=\int_0^t \overrightarrow{v}\cdot \overrightarrow{v}ds,\quad
\langle U, V\rangle(t) :=\int_0^t \overrightarrow{u}\cdot \overrightarrow{v}ds.
\end{gather*}
are called the covariance processes. We can denote
\begin{gather*}
d\langle X, U\rangle(t) := \overrightarrow{x}(t)\cdot \overrightarrow{u}(t),\quad
d\langle X, V\rangle(t) := \overrightarrow{x}(t)\cdot \overrightarrow{v}(t),\\
d\langle Y, U\rangle(t) :=\overrightarrow{y}(t)\cdot \overrightarrow{u}(t),\quad
d\langle Y, V\rangle(t) :=\overrightarrow{y}(t)\cdot \overrightarrow{v}(t),\\
d\langle X, X\rangle(t) :=\overrightarrow{x}(t)\cdot \overrightarrow{x}(t),\quad
d\langle Y, Y\rangle(t) := \overrightarrow{y}(t)\cdot \overrightarrow{y}(t),\\
d\langle X, Y\rangle(t) := \overrightarrow{x}(t)\cdot \overrightarrow{y}(t),\quad
d\langle U, U\rangle(t) :=\overrightarrow{u}(t)\cdot \overrightarrow{u}(t),\\
d\langle V, V\rangle(t) :=\overrightarrow{v}(t)\cdot \overrightarrow{v}(t),\quad
d\langle U, V\rangle(t) :=\overrightarrow{u}(t)\cdot \overrightarrow{v}(t),\\
d\langle Z, Z\rangle(t) := \overrightarrow{x}(t)\cdot \overrightarrow{x}(t)+\overrightarrow{y}(t)\cdot \overrightarrow{y}(t)\,,\\
d\langle W, W\rangle(t) := \overrightarrow{u}(t)\cdot \overrightarrow{u}(t)+\overrightarrow{v}(t)\cdot \overrightarrow{v}(t)\,.
\end{gather*}

The following observations are important. 

\begin{lemma}
\label{locort}
Let $W=U+iV$ be the martingale transform of $Z=X+iY$ by means of the matrix $A$ above.
Then
\begin{equation}
\label{ortheq1}
d\langle U, V\rangle (t) =0\,,\quad d\langle U, U\rangle (t) =d\langle V, V\rangle (t)\,.
\end{equation}
Equivalently,
$$
\overrightarrow{u}(t)\cdot \overrightarrow{v}(t)=0,\quad |\overrightarrow{u}(t)|= |\overrightarrow{v}(t)|\,.
$$ 
\end{lemma}

\begin{lemma}
\label{subord4}
We have
$$
d\langle U, U\rangle (t) \le 2\,d\langle Z, Z\rangle (t)\,,\qquad d\langle V, V\rangle (t) \le 2\,d\langle Z, Z\rangle (t)\,.
$$
Equivalently, 
\begin{gather*}
\overrightarrow{u}(t)\cdot \overrightarrow{u}(t)\le 2\,(\overrightarrow{x}(t)\cdot \overrightarrow{x}(t) + \overrightarrow{y}(t)\cdot \overrightarrow{y}(t))\,,\\
\overrightarrow{v}(t)\cdot \overrightarrow{v}(t)\le 2\,(\overrightarrow{x}(t)\cdot \overrightarrow{x}(t) + \overrightarrow{y}(t)\cdot \overrightarrow{y}(t))\,.
\end{gather*}
or, equivalently,
\begin{equation*}
d\langle W, W\rangle (t) \le 4\,d\langle Z, Z\rangle (t)\,.
\end{equation*}
\end{lemma}

\begin{proof} We have
$$
\overrightarrow{u}(t)\cdot \overrightarrow{u}(t)= (x_1 + y_2)^2 + (x_2-y_1)^2
\le 2\,(x_1^2 + y_2^2 + x_2^2 + y_1^2)= 2\,d\langle Z, Z\rangle\,.
$$
Similar calculations can be made for $v$.
\end{proof}

\begin{defi} The complex martingale $W=A\star Z$ will be called the Ahlfors--Beurling transform of the martingale $Z$.
\end{defi}

Now let us quote a theorem of Banuelos--Janakiraman \cite{BaJ1}:

\begin{theorem}
\label{BaJthMart}
Let $Z,W$ be two martingales on the filtration of the $2$--dimen\-sional Brownian motion, and let $W$ be an orthogonal martingale 
in the sense of \eqref{ortheq1}: $d\langle U, V\rangle(t) =0$, $d\langle U, U\rangle (t) =d\langle V, V\rangle (t)$. 
Suppose that $Z$ and $W$ satisfy the subordination property
\begin{equation*}
d\langle W, W\rangle \le d\langle Z,Z\rangle.
\end{equation*}
Let $p\ge 2$. Then for every $t$,
\begin{equation*}
(\E |W(t)|^p)^{1/p} \le \sqrt{\frac{p^2-p}{2}}(\E |Z(t)|^p)^{1/p}
\end{equation*}
{\rm(}$|\cdot|$ denotes the euclidean norm in $\R^2${\rm)}.
\end{theorem}

One can easily obtain a ``dual" version for $1<p\le 2$:

\begin{theorem}
\label{JVVthMart}
Let $Z,W$ be two martingales on the filtration of the $2$--dimen\-sional Brownian motion, and let $Z$ be an orthogonal martingale 
in the sense of \eqref{ortheq1}: $d\langle X, Y\rangle(t) =0$, $d\langle X, X\rangle (t) =d\langle Y, Y\rangle (t)$. 
Suppose that $Z$ and $W$ satisfy the subordination property
\begin{equation*}
d\langle W, W\rangle \le d\langle Z,Z\rangle.
\end{equation*}
Let $1<p\le 2$. Then for every $t$,
\begin{equation*}
(\E |W(t)|^p)^{1/p} \le \sqrt{\frac{2}{p^2-p}}(\E |Z(t)|^p)^{1/p}\,.
\end{equation*}
\end{theorem}

We use the notations
$$
\|Z(t)\|_p:= (\E |Z(t)|^p)^{1/p}\,.
$$
Sometimes we omit $t$ and just write $\|Z\|_p$.

Theorem \ref{BaJthMart} together with Lemmas \ref{locort}, \ref{subord4}
gives the following statement.

\begin{theorem}
\label{FG1}
$\|W\|_p =\|A\star Z\|_p \le \sqrt{2(p^2-p)} \|Z\|_p$, $p\ge 2$.
\end{theorem}

What happens in Theorems \ref{BaJthMart} for $1<p<2$ and in Theorem \ref{JVVthMart} for $p>2$ is quite interesting, 
especially because these problems have such a close connection to estimates of the Ahlfors--Beurling operator, and because these 
problems exercise a lot of resistance. See also \cite{BaWa1} and \cite[Section 5]{VaVo} where it is shown how a big class of 
singular operators can be obtained from martingale transforms.
\medskip

\noindent{\bf Acknowledgments}. We are very grateful to Vasily Vasyunin who read the preliminary text and made many useful remarks. We are also very grateful to Fedja Nazarov for useful discussion and valuable remarks.

\section{Left-Right orthogonality. Main theorem}
\label{LR}

As above, let us consider two $\R^2$--valued martingales $Z=(X,Y)$ and $W=(U,V)$, both on the filtration of the $2$--dimensional Brownian motion. Using the previous notations, we write
\begin{equation*}
\overrightarrow{x}(t)=\nabla X(t),\quad \overrightarrow{y}(t)=\nabla Y(t),\quad
\overrightarrow{u}(t) =\nabla U(t),\quad \overrightarrow{v}(t)=\nabla V(t)\,.
\end{equation*}
Here the symbol $\nabla$ stands for ``stochastic gradient" of our martingales, somewhat abusing the notations.
We assume the pointwise orthogonality
\begin{equation}
\label{ortheq3}
\overrightarrow{x}(t)\cdot \overrightarrow{y}(t)=0\,,\, |\overrightarrow{x}(t)|=|\overrightarrow{y}(t)|\,,\,
\overrightarrow{u}(t)\cdot \overrightarrow{v}(t)=0\,,\, |\overrightarrow{u}(t)|=|\overrightarrow{v}(t)|\,.
\end{equation}

We also assume the pointwise subordination
\begin{equation}
\label{subordeq3}
|\overrightarrow{u}(t)|^2 +|\overrightarrow{v}(t)|^2\le |\overrightarrow{x}(t)|^2+|\overrightarrow{y}(t)|^2\,.
\end{equation}

To formulate our main result, let us recall that the Legendre function $L_{\alpha}(s)$ of order $\alpha$ is the unique (up to a multiplicative constant) bounded near $1$ solution of the Sturm--Liouville equation
\begin{equation*}
((1-s^2)y'(s))' +\alpha (\alpha +1) y(s) =0\,.
\end{equation*}
It is the hypergeometric function $_2F_1(-\alpha, \alpha+1, 1; \frac{1-s}{2})$.

For any $p\in (2, \infty)$ consider the positive number $\alpha_p$ such that $\alpha_p (\alpha_p +1) =p$ that is
$$
\alpha_p=\frac{\sqrt{1+4p}-1}{2}\,.
$$
Consider the {\bf largest} zero of $L_{\alpha_p}(s)$ on 
$(0,1)$. Denote it by $z_p$.  

\begin{theorem}
\label{MainLR}
Let $p\ge 2$, let $Z,W$ satisfy \eqref{ortheq3} and \eqref{subordeq3}. Then  
\begin{equation*}
\|W\|_p\le \frac{1+z_p}{1-z_p}\,\|Z\|_p\,.
\end{equation*}
Moreover, this constant is sharp.
\end{theorem}

\noindent{\bf Remark.} In a recent preprint of Ba\~nuelos and Osekowski \cite{BO} this theorem is extended to $0<p<2$. Moreover, it is extended to conformal martingales in $\R^d$. The extension is in the language of Bessel processes, which allows for non-integer $d$ as well!

\bigskip

Let $J_0$ be the Bessel function of the zero order,
$$
J_0(x)=\sum_{n\ge 0}(-1)^n\frac{x^{2n}}{2^{2n}(n!)^2}.
$$
Denote its first positive zero by $j_0$. It is known (see, for example, 
\cite[Section 15.51]{WAT}) that
$$
j_0\approx 2.4048.
$$
Furthermore (see Section~\ref{Asymptotics}),
$$
\lim_{p\to\infty}\frac1p\cdot\frac{1+z_p}{1-z_p}=\frac{4}{j^2_0}.
$$
Since $j_0>2\sqrt[4]{2}$, our theorem gives better linear asymptotics 
for the constant as $p\to\infty$ than that in Theorem \ref{BaJthMart}.

To prove Theorem~\ref{MainLR} we are going to introduce the following Bellman function, a variant of Burkholder's function from \cite{Bu1}--\cite{Bu7}, which will work for orthogonal martingales.

\section{The Bellman function}

We consider the martingales given by the stochastic integrals
\begin{gather*}
X(t) = X(0) + \int_0^t \overrightarrow{x}(s)\cdot dB_s\,,\quad 
Y(t) = Y(0) + \int_0^t \overrightarrow{y}(s)\cdot dB_s\,,\\
U(t) = U(0) + \int_0^t \overrightarrow{u}(s)\cdot dB_s\,,\quad 
V(t) = V(0) + \int_0^t \overrightarrow{v}(s)\cdot dB_s\,,\\
M(t) = M(0) + \int_0^t \overrightarrow{m}(s) \cdot dB_s\,.
\end{gather*}
We assume that the martingale $Z(t)=(X(t),Y(t))$ satisfies the condition 
\begin{equation*}
\lim_{t\rightarrow\infty} \E|Z(t)|^p<\infty\,.
\end{equation*}
(Since $\E|Z(t)|^p$ is monotone, we have $\lim_{t\rightarrow\infty} \E|Z(t)|^p=\sup_{t\ge 0} \E|Z(t)|^p$.) 
The martingale $M(t)$ is any martingale majorazing the sub-martingale 
$|Z(t)|^p$. Here we assume 
that the $5\times 2$-matrix of random processes 
$$
e(s):=(\overrightarrow{x}(s),\overrightarrow{y}(s), \overrightarrow{u}(s), \overrightarrow{v}(s), \overrightarrow{m}(s))^T
$$ 
satisfies pointwise the condition $e(s) \in A$, where 
\begin{gather*}
A=\{e\in M_{5\times 2}: e_{11}e_{21}+ e_{12}e_{22}=0, e_{11}^2+e_{12}^2 =e_{21}^2+e_{22}^2, e_{31}e_{41}+ e_{32}e_{42}=0,\\ 
e_{31}^2+e_{32}^2 =
e_{41}^2+e_{42}^2,  e_{31}^2+e_{32}^2 +e_{41}^2+e_{42}^2\le e_{11}^2+e_{12}^2 +e_{21}^2+e_{22}^2\}\,.
\end{gather*}
The random process $e(s)$ is also assumed to be a non-anticipatory process, that is, $e(s)$ is measurable with respect to
the $\sigma$-algebra generated by the $2$--dimensional Brownian motion $\{B_{\tau}\,,\, \tau \le s\}$.

Introduce now $W=(U,V)$. For $T=(T_1,T_2,T_3,T_4,T_5)\in\mathbb R^5$ 
we define
$$
\B(T):= \hskip -2 cm\sup_{\vbox{\hskip 2 cm\scriptsize $\text{all  non-anticipatory processes}\,\, e_{ij} \,\,\text{such that} \,\,e(s)\in A,$
\vskip 0.04 cm\hskip 1.2 cm 
$X(0)=T_1$, $Y(0)=T_2$, $U(0)=T_3$, $V(0)=T_4$, $M(0)=T_5$, $M(t)\ge |Z(t)|^p$}}\hskip -3 cm\{\lim_{t\rightarrow\infty} \E|W(t)|^p\}\,.
$$

It is convenient to change the notations and write everything in the following more compact form of a motion in $\R^5$:
$$
R(t)=R(0) +\int_0^t e(s) dB_s\,.
$$
Then
$$
\B(T):= \hskip -3 cm\sup_{\vbox{\hskip 2.2 cm\scriptsize $\text{all  non-anticipatory processes}\,\, e_{ij} \,\,\text{such that} \,\,e(s)\in A,$
\vskip 0.04 cm\hskip 3.6 cm 
$R(0)=T$, $R_5(t)\ge (R_1(t)^2+ R_2(t)^2)^{p/2}$}}
\hskip -3.7 cm\{\lim_{t\rightarrow\infty} \E(R_3(t)^2+ R_4(t)^2)^{p/2}\}\,.
$$
The function $\B(T)$ is defined on the following convex domain inside $\R^5$:
$$
\Omega:=\{T\in \R^5: T_5 \ge (T_1^2 +T_2^2)^{p/2}\}\,.
$$

\subsection{Properties of the Bellman function}
\label{Bellf}

Let us fix a positive time $t$. 
Choose any non-anticipatory process $e(\tau),\, 0\le \tau < t\,,$ satisfying the above restrictions. 
If we start from a point $R(0)\in \Omega$, we obtain the points 
$P_t=R(t,\omega)$. These are our starting data now.

Choose a matrix process $e(s)\,, s\ge t$, for a given $R(t, \omega)$, that attains the supremum in the definition of
$\B(R(T))$ up to a small $\varepsilon>0$. Then the process $e$, equal to  
$e(\tau)$, $0\le \tau < t$, $e(s)$, $s>t$, should be compared to the processes for the starting  data $T=R(0)$
giving the supremum in the definition of $\B(R(0))$.  Let us do this comparison.
Introduce 
$$
F(T)=(T_3^2 +T_4^2)^{p/2}\,.
$$
Let $R(t)$ be the martingale driven by  $e$ constructed above.
Using the formula of full probability and stationarity of Brownian motion  $B_s$ we can write ``Bellman's principle":
$$
\B(R(0))\ge \E F(R(\infty)) = \E \E( F(R(\infty))| R(t,\omega) =P_t)\ge\E\B(R(t))-\varepsilon\,.
$$
In other words,
$$
\E (\B(R(t))-\B(R(0)))\le 0\,.
$$
Since $F$ is convex, and $R(t)$ satisfies the martingale property for any initial point $T=R(0)$ in $\Omega$, 
we have obviously 
\begin{equation*}
\B(T)\ge \B(R(t)) \ge \E F(R(\infty)) \ge F(\E R(\infty))= F(R(0)) = (T_3^2+T_4^2)^{\frac{p}2}\,.
\end{equation*}
Now we apply the It\^o formula for the difference:
$$
\E (\B(R(t))-\B(R(0)))=\frac12\int_0^t\E \sum_{k,l=1}^5 \frac{\partial^2 \B}{\partial T_k\partial T_l} \,(dR_k(s)\cdot dR_{l}(s))\,ds\,;
$$
we use here that $\E |B_{s+\Delta s}-B_s|^2 =\Delta s$.
Note that by the formulas at the beginning of the section, $dR_j(s)$ is exactly $e_{j}(s)$, that is the $j$-th row of the matrix $e$. 
Therefore,
\begin{gather*}
\E (\B(R(t))-\B(R(0)))=\frac12\int_0^t\E \sum_{k,l=1}^5 \frac{\partial^2 \B}{\partial T_k\partial T_l}\,( \overrightarrow{e_k}(s)\cdot \overrightarrow{e_l}(s))\,ds \\
=\frac12 \int_0^t\E \,\text{trace}\, (e(s)^{T} d^2 \B \,e(s))\, ds\le 0\,,
\end{gather*}
where $d^2 \B$ denotes the Hessian (matrix) of $\B$.

This formula holds for all non-anticipatory matrix processes $e(s)$, $0\le s\le t$. 
We have already tacitly assumed the smoothness of $\B$. Using this assumption again, we divide the latter inequality 
by $t$ and pass to the limit $t\rightarrow 0$. Then we get:
\begin{equation}
\label{Bellineq}
-\text{trace}\, (e(s)^{T} d^2 \B\, e(s))=-\sum_{k,l=1}^5 \frac{\partial^2 \B}{\partial T_k\partial T_l}\, (\overrightarrow{e_k}\cdot \overrightarrow{e_l})\ge 0\,\qquad e\in A\,.
\end{equation}

Actually we might hope to have more (and these hopes will be, although only partially, fulfilled):
\begin{equation}
\label{Belleq}
\max_{e\in A, e\neq 0}\sum_{k,l=1}^5 \frac{\partial^2 \B}{\partial T_k\partial T_l}\,( \overrightarrow{e_k}\cdot \overrightarrow{e_l})=0,.
\end{equation}
We will not use \eqref{Belleq} in the future. As we told, it was just a hope.

A more rigorous analysis of how to obtain \eqref{Belleq} can be found in \cite{K}. 
Still, it is not totally clear what conditions guarantee that ``each state has the best control". 

Note that each vector $e_j, j=1,\ldots,5$ has two coordinates. Let us unite all first coordinates 
and call the corresponding $5$-vector $e^1$, similarly we get $e^2$. 
Then \eqref{Bellineq} can be rewritten as
\begin{equation*}
\text{trace}\, (e^{T} d^2 \B\, e)=(d^2\B\, e^1)\cdot e^1 + (d^2\B\, e^2)\cdot e^2 \le 0\,,\qquad e\in A\,.
\end{equation*}

Similarly, we might hope to have
\begin{equation*}
\max_{e\in A, e\neq 0
}\text{trace}\, (e^{T} d^2 \B\, e)=\sup_{e=(e^1, e^2)\in A
}((d^2\B\, e^1)\cdot e^1 + (d^2\B\, e^2)\cdot e^2 )=0\,.
\end{equation*}


\begin{theorem}
\label{above}
Let $\Psi$ be any function 
such that
\begin{equation}
\text{\rm trace}\, (e(s)^{T} d^2 \Psi\, e(s))\le 0\,,\qquad e\in A \label{190}
\end{equation}
(in the sense of distributions),
\begin{equation}
\Psi(T) \ge F(T)\,,
\label{obstacle11}
\end{equation}
and let for some $c>0$ we have
\begin{equation}
\label{191}
\Psi(T)\le c^p T_5\,,
\end{equation}
for all $T$ satisfying $T_1=T_2=T_3=T_4=0$.

Then $\|W(t)\|_p\le c \|Z(t)\|_p$, for any time $t$ and any two orthogonal martingales $W, Z$ such that 
$W$ is differentially subordinated to $Z$ in the sense of \eqref{subordeq3}.
\end{theorem}

\begin{proof}
Fix $t\ge 0$ and denote $R_1:=X$, $R_2:=Y$, $R_3:=U$, $R_4:=V$, 
$M:=|(R_1(t), R_2(t))|^p$. Consider the martingale 
$R_5(s) = \E(M|\mathcal{F}_s)$, 
where $\mathcal{F}_s$ is the $\sigma$-algebra generated by $B_{1}(\tau), B_{ 2}(\tau)$ for $\tau\le s$. 
Let us consider $ \E \Psi(R(s))$. If we apply It\^o's formula to 
$\Psi(R(s))$ on $[0,t]$ and take the expectation, we get
$$
\E\Psi(R(t))=\frac12\E\int_0^t \sum_{k,l=1}^5 \frac{\partial^2 \Psi }{\partial T_k\partial T_l} (dR_k(s) \cdot dR_l(s))\,ds
+\Psi(R(0))\,.
$$
Next we use \eqref{190}--\eqref{191} and the convention $R_3=U$, $R_4 =V$ to get
$$
\E|(U(t), V(t))|^p \le \Psi(R(0)) \le c^p  R_5(0)\,.
$$
Since $R_5(s)$ is a martingale, we have 
$$ 
R_5(0) =\E R_5(0) =\E R_5(t) =\E M= \E|(R_1(t), R_2(t))|^p\,.
$$
It remains to note that we have the convention: $R_1=X$, $R_2=Y$.
\end{proof}

\begin{cor}
\label{192}
The best constant $c$ such that $\B(T)\le c^p T_5$ for all $T$ satisfying $T_1=T_2=T_3=T_4=0$ 
coincides with the best constant $c$ such that $\|W(t)\|_p\le c \|Z(t)\|_p$, for any time $t$ and any two orthogonal martingales $W, Z$ such that 
$W$ is differentially subordinated to $Z$.
\end{cor}

It is easy to see several other properties of $\B$. For example, by multiplying our martingales $T_1,T_2,T_3,T_4$ by 
a positive constant $\tau$ one obtains:
\begin{equation*}
\B(\tau T_1,\tau T_2,\tau T_3,\tau T_4,\tau^p T_5)=\tau^p\B(T)\,.
\end{equation*}

Let $S$ be a unitary operator on $\R^2$. Note that if we multiply its matrix by a matrix
$$
\begin{bmatrix} 
x_1 & x_2\\
y_1 & y_2
\end{bmatrix}
$$
such that its rows are orthogonal and the norms of rows are equal, then we have again a matrix 
with orthogonal rows having the same norm.

For us this means that given $(X,Y)^T$, $(U, V)^T$ (or $(R_1, R_2)^T, (R_3, R_4)^T$ in other notations), 
we can apply $S$ to these vectors, and we can multiply the corresponding martingales for any $t$ by 
the constant unitary matrix of $S$. Note that the process $e$ will be transformed. Namely, the fifth row stays the same, 
but the rows $1,2$ and the rows $3,4$ form matrices which are multiplied on the left by the matrix of $S$. 
As we have found out, the new matrix $e'$ has the same properties of rows as $e$! So again $e'\in A$ pointwise. 
Of course, $|S\cdot (U,V)^{T}(t)|=|(U,V)^T(t)|$ pointwise. This reasoning gives us the following property of $\B$:
\begin{equation*}
\B(T)=: b((T_1^2+T_2^2)^{1/2}, (T_3^2+T_4^2)^{1/2}, T_5)\,.
\end{equation*}

The next property becomes obvious when we take $R_1(t)=R_3(t), R_2(t)=R_4(t)$:
\begin{equation*}
(T_1^2+T_2^2)^{p/2} \le \B(T)\,.
\end{equation*}

\subsection{Reduction of the number of variables}
\label{reduction}

Fix some $c>0$ and $T'=(T_1, T_2, T_3, T_4)$. Recall that $(T', T_5)\in \Omega$ if and only if $T_5\ge (T_1^2+T_2^2)^{p/2}$, 
and consider
$$
\Phi(T') =\Phi_c(T') =\sup_{T=(T', T_5): T\in \Omega} (\B(T) -c^p T_5)\,.
$$

If $c$ is larger than the best constant in Corollary~\ref{192}, then $\Phi$ is well defined at $0$, and hence, 
everywhere.

Obviously, we have
\begin{gather}
\label{homogPhi}
\Phi(tT')=t^p\Phi(T')\,,\\
\label{rotatPhi}
\Phi(T')=: \phi((T_1^2+T_2^2)^{1/2}, (T_3^2+T_4^2)^{1/2})\,.
\end{gather}

By the submartingale property of $|W(t)|^p$ we obtain:
\begin{equation*}
(T_3^2+T_4^2)^{p/2} \le \B(T)\,.
\end{equation*}

What is much less easy (but still true) is that the concavity in the sense of \eqref{Bellineq} is also preserved. 
It is much less easy because the supremum of concave functions is not obliged to be concave. But if we have 
a concave function of several variables and form a new function which is the supremum of the original function 
over one of the variables, then the result is concave again. The same reasoning gives  
\begin{equation}
\label{BellineqPhi}
-\text{trace} \,(e'^{T} d^2\Phi\, e')=-\sum_{k,l=1}^4 \frac{\partial^2 \Phi}{\partial T_k\partial T_l} \,(\overrightarrow{e_k}\cdot \overrightarrow{e_l})\ge 0\,,\qquad e\in A\,.
\end{equation}
Here $e'$ denotes the matrix $e$ with deleted fifth row.

Inequality \eqref{BellineqPhi} can be rewritten (again, in the sense of distributions) as
\begin{equation}
\label{Bellineq1Phi}
\text{trace} \,(e'^{T} d^2\Phi\, e')=(d^2\Phi\, e^1)\cdot e^1 + (d^2\Phi\, e^2)\cdot e^2 \le 0\,,\qquad e\in \tilde A\,.
\end{equation}
Here $e^1, e^2$ are $4$-vectors, namely $e^1$ is the (column) vector of the first coordinates of all vectors $\overrightarrow{e_k}$, $k=1,2,3,4$, and $e^2$ is the (column) vector of the second coordinates of all vectors $\overrightarrow{e_k}$, $k=1,2,3,4$. 
We write $e^1=(h,k)^T, e^2=(h', k')^T$, where $h, k, h', k'$ are (row) $2$-vectors.
Furthermore, the conditions on $e'$ are the same as those on $e$ but now with the fifth row deleted. We call 
these conditions $\tilde{A}$, and here they are: 
\begin{equation}
\label{tildeA}
|h|=|h'|,\quad |k|=|k'|, \quad h\cdot h'=0, \quad k\cdot k'=0,\quad |k|\le |h|\,.
\end{equation}

The next property is obvious:
\begin{equation}
\label{x02}
(T_3^2+T_4^2)^{p/2} -c^p\,(T_1^2+T_2^2)^{p/2} \le \Phi(T)\,.
\end{equation}

In the opposite direction, starting with a function $\Phi$ and $c>0$ satisfying 
\eqref{homogPhi}, \eqref{rotatPhi}, \eqref{Bellineq1Phi}, \eqref{x02}, we can define 
$$
\Psi(T)=\Phi(T')+c^5T_5,
$$
and apply Theorem~\ref{above}.

\begin{cor}
The best constant $c$ such that there exists a function $\Phi$ satisfying 
\eqref{homogPhi}, \eqref{rotatPhi}, \eqref{Bellineq1Phi}, \eqref{x02}
coincides with the best constant $c$ such that $\|W(t)\|_p\le c \|Z(t)\|_p$, for any time $t$ and any two orthogonal martingales $W, Z$ such that 
$W$ is differentially subordinated to $Z$.
\end{cor}

\subsection{Another reduction of the number of variables}

The function $\Phi$ has $4$ variables, but the radial symmetry allows us to reduce it to a function of  
only $2$ variables. Namely, using \eqref{homogPhi} and \eqref{rotatPhi} we obtain a function $\phi(x,y)$, 
$x\ge 0$, $y\ge 0$ such that
\begin{equation}
\label{homogphi}
\phi(tx,ty)=t^p\phi(x,y)\,.
\end{equation}

By \eqref{x02}, we have 
\begin{equation*}
y^p -c^p x^p \le \phi(x,y)\,.
\end{equation*}

Now we want to rewrite \eqref{Bellineq1Phi} in terms of $\phi$ using \eqref{tildeA}. Set 
$$
z=\frac{(T_1, T_2)}{|(T_1, T_2)|},\qquad w=\frac{(T_3, T_4)}{|(T_3, T_4)|}.
$$ 
Given a vector $h$, we denote by $h^{\perp}$ the projection of $h$ on 
the direction orthogonal to $z$; given a vector $k$ we denote by $k^{\perp}$ the projection of $k$ on 
the direction orthogonal to $w$,  the same with $h', k'$. 

Set $x=(T_1^2+T_2^2)^{1/2}$, $y=(T_3^2+T_4^2)^{1/2}$, $u_1=u_2=x$, $u_3=u_4=y$. We have
\begin{gather*}
\frac{\partial u_j}{\partial T_j}=\frac{T_j}{u_j},\quad
\frac{\partial \Phi}{\partial T_j}=\frac{\partial \phi}{\partial u_j}\cdot\frac{T_j}{u_j},\qquad 1\le j\le 4,\\
\frac{\partial^2 \Phi}{\partial T_j^2}=\frac{\partial^2 \phi}{\partial u_j^2} \cdot\frac{T_j^2}{u_j^2}+
\frac{\partial \phi}{\partial u_j} \cdot\frac{u_j^2-T_j^2}{u_j^3},\qquad 1\le j\le 4,\\
\frac{\partial^2 \Phi}{\partial T_j\partial T_k}=
\frac{\partial^2 \phi}{\partial u_j^2} \cdot\frac{T_jT_k}{u_j^2}-
\frac{\partial \phi}{\partial u_j} \cdot\frac{T_jT_k}{u_j^3},
\qquad 1\le j,k\le 4,\,u_j= u_k,\,j\not= k,\\ 
\frac{\partial^2 \Phi}{\partial T_j\partial T_k}=\frac{\partial^2 \phi}{\partial x\partial y} \cdot\frac{T_jT_k}{xy},
\qquad 1\le j,k\le 4,\,u_j\not= u_k.
\end{gather*}

Then  
\begin{gather*}
(d^2\Phi \,(h,k)^T)\cdot (h,k)^T + (d^2\Phi \,(h',k')^T)\cdot (h',k')^T=\\
\frac1{x}\frac{\partial \phi}{\partial x} |h^{\perp}|^2 + \frac{\partial^2\phi}{\partial x^2} (h\cdot z)^2 +2\frac{\partial^2\phi}{\partial x\partial y}(h\cdot z)(k\cdot w) +\frac{\partial^2\phi}{\partial y^2}(k\cdot w)^2 +
\frac1{y}\frac{\partial \phi}{\partial y} |k^{\perp}|^2+\\
\frac1{x}\frac{\partial \phi}{\partial x} |h'^{\perp}|^2 + \frac{\partial^2\phi}{\partial x^2} (h'\cdot z)^2 +2\frac{\partial^2\phi}{\partial x\partial y}(h'\cdot z)(k'\cdot w) +\frac{\partial^2\phi}{\partial y^2}(k'\cdot w)^2 +
\frac1{y}\frac{\partial \phi}{\partial y} |k'^{\perp}|^2\\=
\frac1{x}\frac{\partial \phi}{\partial x} |h|^2 + \frac{\partial^2\phi}{\partial x^2} |h|^2 +2\frac{\partial^2\phi} {\partial x \partial y}(h\cdot Sk) +\frac{\partial^2\phi}{\partial y^2}|k|^2 +
\frac1{y}\frac{\partial \phi}{\partial y} |k|^2\,.
\end{gather*}
Here $S$ stands for a unitary operator sending $w$ to $z$. Note that the last expression must be non-positive 
for any $h,k$ such that 
\begin{equation*}
|k|\le |h|\,.
\end{equation*}

Thus, \eqref{Bellineq1Phi} becomes (in the sense of distributions)
\begin{equation}
\label{Bellineq1phi}
\frac1{x}\frac{\partial \phi}{\partial x} |h|^2 + \frac{\partial^2\phi}{\partial x^2} |h|^2 
+2\frac{\partial^2\phi}{\partial x \partial y}(h\cdot k) +\frac{\partial^2\phi}{\partial y^2}|k|^2 +
\frac1{y}\frac{\partial \phi}{\partial y} |k|^2\le 0\,,\quad |k|\le |h|\,.
\end{equation}

\section{Reduction to differential inequalities in one dimension}
\label{DI} 

We can reduce our problem to the following one. We are looking for the smallest $c$ such that the function
$$
h_c(x,y) := y^p-c^px^p
$$ 
can be majorized by a solution of the differential inequality \eqref{Bellineq1phi} satisfying \eqref{homogphi}.
Both $\phi$ and $h_c$ are $p$-homogeneous. Therefore, we can further reduce our problem to that on functions of one real variable.
First note that \eqref{Bellineq1phi} is equivalent to the fact that a certain quadratic polynomial is negative when its argument is bigger than $1$ in absolute value. This is equivalent to three inequalities, the first and the second of which are
\begin{align}
\label{minus}
\frac1{x}\frac{\partial \phi}{\partial x}  + \frac{\partial^2\phi}{\partial x^2}  -2\frac{\partial^2\phi}{\partial x \partial y} +\frac{\partial^2\phi}{\partial y^2}+
\frac1{y}\frac{\partial \phi}{\partial y} &\le 0\,,\\
\label{plus}
\frac1{x}\frac{\partial \phi}{\partial x}  + \frac{\partial^2\phi}{\partial x^2}  +2\frac{\partial^2\phi}{\partial x \partial y} +\frac{\partial^2\phi}{\partial y^2}+
\frac1{y}\frac{\partial \phi}{\partial y} &\le 0\,.
\end{align}
These relations claim just that our quadratic expression is negative when its argument is equal to $\pm 1$.

Let $\D$ denote the discriminant  of our quadratic expression,
$$
\D=\Bigl( \frac{\partial^2\phi}{\partial x\partial y}\Bigr)^2-
\Bigl(\frac1{x}\frac{\partial \phi}{\partial x}+\frac{\partial^2\phi}{\partial x^2}\Bigr)
\Bigl(\frac1{y}\frac{\partial \phi}{\partial y}+\frac{\partial^2\phi}{\partial y^2}\Bigr)\,.
$$
Clearly, if $\D<0$, and if the quadratic polynomial is negative at $\pm 1$ by \eqref{minus}, \eqref{plus}, 
then it is negative for all the arguments exceeding $1$ in absolute value and \eqref{Bellineq1phi} holds. 
If $\D\ge 0$, then \eqref{Bellineq1phi} follows from \eqref{minus}, \eqref{plus}, and the fact that 
the smaller root of the quadratic expression belongs to $(-1,1)$, which is our third inequality:
\begin{equation}
\label{middle}
\bigg|\Bigl|\frac{\partial^2\phi}{\partial x\partial y}\Bigr|-\D^{1/2}\bigg|\le \bigg|\frac{\partial^2\phi}{\partial x^2}+
\frac1{x}\frac{\partial \phi}{\partial x}\bigg|\,.
\end{equation}

Now using homogeneity we write
\begin{gather}
s:= \frac{y-x}{x+y}\,,\quad \frac{y}{x+y}=\frac{1+s}{2}\,,\quad \frac{x}{x+y}=\frac{1-s}{2}\,,\label{x03}\\
\phi(x,y) =(x+y)^p \phi\Bigl(\frac{x}{x+y}, \frac{y}{x+y}\Bigr) = (x+y)^p \phi\Bigl(\frac{1-s}{2}, \frac{1+s}{2}\Bigr)\,,\notag\\
g(s):=\phi\Bigl(\frac{1-s}{2}, \frac{1+s}{2}\Bigr)\,.\notag
\end{gather}
 
On $\{(x,y):x+y=1\}$ we have:
\begin{gather}
\phi_{xx}=p(p-1)g(s)-2(p-1)(1+s)g'(s)+(1+s)^2g''(s)\,,\notag\\
\phi_{yy}=p(p-1)g(s)+2(p-1)(1-s)g'(s)+(1-s)^2g''(s)\,,\notag\\
\phi_{xy} =p(p-1) g(s)-2(p-1) s g'(s) -(1-s^2) g''(s)\,,
\\
\frac{\phi_{x}}{x}=\frac{2p}{1-s} g(s) - \frac{2(1+s)}{1-s} g'(s)\,, 
\\
\frac{\phi_{y}}{y}=\frac{2p}{1+s} g(s) + \frac{2(1-s)}{1+s} g'(s)\,, 
\\
\phi_{xx} - 2 \phi_{xy} +\phi_{yy} =4 g''(s)\,. 
\end{gather}
 
Denote
\begin{gather*}
Kg(s):=p(p-1) g(s)-2(p-1) s g'(s) -(1-s^2) g''(s)\,,\\
Dg(s):= ((1-s^2)g'(s))' +pg(s)\,,\\ 
\tilde{D}g(s):=\frac{Dg(s)}{1-s^2} +Kg(s)\,.
\end{gather*}
 
Then \eqref{minus}, \eqref{plus} can be rewritten correspondingly  as
\begin{gather}
Dg(s) = (1-s^2)g''(s) -2s g'(s) + pg(s) \le 0\,, \qquad s\in [-1,1]\,,\label{minus1} \\
\tilde{D}g(s) = \frac{Dg(s)}{1-s^2} +Kg(s) \le 0\,, \qquad s\in [-1,1]\,.\label{plus1}
\end{gather}

Next we pass to \eqref{middle}. Suppose that $Dg=0$ on an interval $I$.
Then  we can calculate on $\{(x,y):x+y=1,\,\frac{y-x}{x+y}\in I\}$:
\begin{gather*}
\phi_{xy} = p(-2sg'+pg)\,,\\ 
\frac{1}{x} \phi_x +\phi_{xx}= p\bigl(-2(1+s) g' +pg\bigr)\,,\quad \frac{1}{y} \phi_y +\phi_{yy}= p\bigl( 2(1-s) g' +pg\bigr)\,.
\end{gather*}
Therefore, 
$$
\D= 4 p^2g'^2\ge 0\,,
$$
and condition \eqref{middle} becomes 
\begin{equation}
\label{BDA}
\bigl||-2sg'+pg|-2|g'|\bigr|\le |-2(1+s)g'+pg|\,,
\end{equation}
which is just the triangle inequality.

Thus, for any interval $I$ where $g$ is a solution of the Legendre equation $Dg=0$,  
\eqref{middle} is automatically satisfied on $\{(x,y):\frac{y-x}{x+y}\in I\}$.

\section{Case $p=n(n+1)$}
\label{int}

From now on, we assume that $p=\alpha(\alpha+1)$, $\alpha>1$, and
we use the notation $D_{\alpha}$ for the operator $D$ defined in \eqref{minus1}.

If $\alpha=n \in \mathbb{N}$, then $p=n(n+1)$ and the Legendre functions $L_n$ (i.e. the solutions of 
the equation $D_n L_n=0$ bounded near the point $1$) become Legendre polynomials
(see, for instance, \cite[Section 3.8]{C}):
\begin{equation}
\label{Ln}
L_{n}(s) =\frac{1}{2^n n!} \frac{d^n}{ds^n} (s^2-1)^n,\qquad L_{n,a} = a\, L_n\,.
\end{equation}
Note that $L_n(1)=1$.

Let us consider an obstacle function
$$
h_c(s):=\Bigl(\frac{1+s}{2}\Bigr)^p-c^p\Bigl(\frac{1-s}{2}\Bigr)^p\,.
$$
Our first remark is that
\begin{equation*}
D_{\alpha} h_c(s) = D_{\alpha} \Bigl(\Bigl(\frac{1+s}2\Bigr)^p-c^p \Bigl(\frac{1-s}2\Bigr)^p\Bigr) = 
\frac{p}{p-1} (1-s^2) h_c''(s)\,.
\end{equation*}

In particular, the inflection point $i_p$ of $h_c$ coincides with the point where the function $D_\alpha h_c$
changes the sign from positive to negative (when we move from $s=1$ to $s=-1$).

Suppose that $\alpha=n> 1$ is an integer. If $p=n(n+1)$, then the Legendre equation $D_n g=0$ has two linearly independent solutions: one is the Legendre polynomial $L_n$ of degree $n$, and another, $Q_n$, has logarithmic singularities at $x=\pm 1$. 

The following statement is a partial case of Lemma~\ref{minzero} we prove later on:

\begin{lemma}
\label{solutions}
Consider the set $Y_p$ of all linear combinations of these solutions $L_n$ and $Q_n$. For any $y\in Y_p$, let $f(y,p)$ 
be the largest zero of $y$ on $[-1,1]$. Then
\begin{equation*}
\min_{y\in Y_p} f(y,p) = f(L_n, p)=: z_p\,.
\end{equation*}
\end{lemma}

Consider all the pairs $(f, h_c)$, where $f$ is a solution of the Legendre equation such that $f(1)>1$ and $h_c$ is as above (recall that $h_c(1)=1$), $f>h_c$ on some interval $(x, 1]$, and $f$ and $h_c$ have the same values and the same derivatives at $x$. Here is a question which will occupy our attention almost till the end of this section:

\medskip

\noindent{\bf Question:} What is the smallest $c$ possible for such pairs?

\begin{theorem}
\label{touching}
The smallest $c$ is $\frac{1+z_p}{1-z_p}$.
\end{theorem}

Consider the following situation: $L_{n,a}$ and $h_c$ touch at a certain point $x=x(p,a)$, the first function being above 
the second one on $(x(p,a),1]$. It is easy to obtain such a situation. Fix $a>1$ to have $L_{n,a}(1) >1$. 
Note that $h_c(1)=1$ for every $c$. Take sufficiently large $c$. Since $z_p$ is the first zero of $L_n$ 
(counting from $1$ to the left, $p=n(n+1)$), we have $L_{n,a} > h_c$ on $[z_p, 1]$. Start to decrease $c$. 
At a touching point $x=x(p, a), c=c(p,a)$ we have the equations 
$$
\begin{cases}(\frac{1+x}2)^p-c^p (\frac{1-x}2)^p= aL_n(x)\,,\\
\frac{p}2((\frac{1+x}2)^{p-1}+c^p (\frac{1-x}2)^{p-1}) = aL_n'(x)\,.\end{cases}
$$
 
Then
\begin{gather*} 
\Bigl(\frac{1+x}{2}\Bigr)^{p-1} =a\bigl(\frac{1}{p}(1-x)L_{n}'(x) + L_n(x)\bigr),\\
c^p\Bigl(\frac{1-x}{2}\Bigr)^{p-1} =a\bigl(\frac{1}{p}(1+x)L_{n}'(x) - L_n(x)\bigr)\,,
\end{gather*}
and we get
$$
c^p =\Bigl(\frac{1+x}{1-x}\Bigr)^{p-1} \frac{\frac1p(1+x) L_n'(x) -L_n(x)}{\frac1p(1-x) L_n'(x) +L_n(x)}\,.
$$

Let us consider the function
$$
\beta(x):=  \frac{(1+x)^p L_n'(x) -p(1+x)^{p-1}L_n(x)}{(1-x)^p L_n'(x) +p(1-x)^{p-1}L_n(x)}.
$$ 

\begin{lemma}
\label{zero}
The function $\beta$ is strictly increasing on the interval $[z_p, 1)$.
\end{lemma}
 
\begin{proof}
Let us differentiate $\beta$ and use the fact that $L_n$ satisfies the equation $(1-x^2) L_n''(x) -2x L_n'(x) +pL_n(x)=0$.
Then we get
\begin{gather} 
\frac{\beta'(x)((1-x)^p L_n'(x) +p(1-x)^{p-1}L_n(x))^2}{ 2p(1-x^2)^{p-2}}\notag\\
= L_n(x) ( (1-x^2) L_n''(x) + (p-1) \cdot 2x L_n'(x) -p(p-1) L_n(x))\notag\\
= p(1-x^2)\,L_n(x)\,L_n''(x)\,.
\end{gather} 
The only zero of $L_n$ on $[z_p,1)$ is $z_p$. The polynomial $L_n$ is positive on $(z_p,1)$. Let us show that  
\begin{equation}
\label{convzp}
L_n''(s)>0\,,\qquad s\in[z_p,1)\,.
\end{equation}  
The orthogonal polynomial $L_n$ has exactly $n$ zeros on $[-1,1]$ (it is a general property, but it also follows easily 
from formula \eqref{Ln}). Then $L'_n$ has exactly $n-1$ zeros on $[-1,1]$, and its largest zero is to the left of $z_p$. 
Furthermore, $L_n''$ has exactly $n-2$ zeros on $[-1,1]$, and its largest zero is also to the left of $z_p$. Since $L_n$ 
is zero at $z_p$ and one at $1$, we conclude that the sign of $L_n''$ on $[z_p,1)$ is positive, which proves \eqref{convzp}.
 
Thus, the function $\beta$ is strictly increasing on $[z_p,1)$. 
\end{proof}
 
We continue by defining
$$
a(x):=\frac{p}2\Bigl(\bigl(\frac{1+x}2\bigr)^{p-1}+\beta(x) \bigl(\frac{1-x}2\bigr)^{p-1}\Bigr)L_n'(x)^{-1}\,,\qquad x\in [z_p,1).
$$
Then $(x, a(x),\beta(x)^{1/p})$, $x\in (z_p,1)$, is a {\it touching triple}, in the sense that $a(x)L_n$ touches 
$h_{\beta(x)^{1/p}}$ at the point $x$. We have already checked that $\beta$ increases strictly on $[z_p,1)$ from 
$\bigl(\frac{1+z_p}{1-z_p}\bigr)^p$ to $+\infty$. Let us check now that $a(x)$ decreases strictly on $[z_p,1)$
from 
\begin{equation}
\label{nakrayu}
a(z_p)= h'_{\frac{1+z_p}{1-z_p}}(z_p)/L_n'(z_p)
\end{equation}
to $1$.
 

\begin{lemma}
\label{adecr}
The function $a(x)$ is strictly decreasing on $[z_p,1)$.
\end{lemma}

\begin{proof}
At the touching point, we have
$$ 
a(x)L_n(x) = h_{\beta(x)^{1/p}}(x)\,.
$$
Differentiating both sides gives us 
$$
a'(x)L_n(x)+a(x)L_n'(x) = h_{\beta(x)^{1/p}}'(x)-\beta'(x)\Bigl(\frac{1-x}{2}\Bigr)^p\,.
$$
Since $h_{\beta(x)^{1/p}}'(x) = a(x)L_n'(x)$, we obtain
$$
a'(x)L_n(x)=-\beta'(x)\Bigl(\frac{1-x}{2}\Bigr)^p= -A(x) L_n(x)L_n''(x), 
$$
where $A(x)>0$.
Since $L_n>0$ in $(z_p,1)$ and $L_n''>0$ in $[z_p,1)$, the assertion of the lemma follows.
\end{proof}

Summing up, we have unique touching triples $(x,a(x),\beta(x)^{1/p})$ 
for $x\in [z_p,1)$ with 
$a(x)\to1$ and $\beta(x)\to\infty$ as $x\to1$. Let $(z_p, a(z_p), \beta(z_p)^{1/p})$ be the triple at $z_p$. Choose $c$ large enough 
so that $h_c<a(z_p) L_n$ on the whole interval $[z_p,1]$. Now decrease $c$ continuously till $h_c$ ($h_c$ increases)  
first meets $a(z_p)L_n$ at some point $x\in [z_p,1)$. 

If the meeting point $x$ is on $(z_p,1)$, it must be also a touching point. Then $a(x)=a(z_p)$. On the other hand,
$x>z_p$, and so by Lemma \ref{adecr} one has $a(x) <a(z_p)$. We came to a contradiction.
Thus, the meeting point is $z_p$. Next we verify that it is also a touching point. 
Indeed, $c$ is equal to $\frac{1+z_p}{1-z_p}$, therefore $h_c'(z_p) = a(z_p) L_n'(z_p)$ by \eqref{nakrayu}.
This means exactly that $h_c$ (with this $c$) not only meets $a(z_p)L_n$ at $z_p$ but also touches it at $z_p$.
Finally, $h_c$ stays below $a(z_p) L_n$ on the whole $(z_p,1)$: that is how we constructed that $c$.
This completes the proof of Theorem \ref{touching}.
\bigskip
 
Here is a different proof of Theorem \ref{touching}. Let us start with several remarks. First of all, 
let us fix $1<a<a(z_p)$ 
and starting with $c=\infty$ let us decrease $c$ until the first time $h_c(x)$ 
meets $L_{n,a}(x)$ on $[z_p,1)$ being below $L_{n,a}(x)$ on the whole interval $[z_p,1)$.
Since $a<a(z_p)$, the meeting points are on $(z_p,1)$, and hence are touching points.

The touching point on the interval $[z_p,1)$ (we call it $x(a)=x(n,a)$) is unique. Otherwise, suppose 
that for some $a>1$ and for some $c=c(a)$ we have two touching points on $[z_p, 1)$. Then 
by the previous considerations we have
\begin{equation*}
\beta(x_1)=\beta(x_2) =c^p\,,
\end{equation*}
which contradicts to Lemma \ref{zero}. We increase $a$ and get decreasing $c(a)$ 
and touching points $x(a)\in [z_p,1)$. How can this process end up? The first possibility is that for 
a certain $a_p>1$ we have $h_c$ lying below $L_{n, a_p}$ and touching it at $z_p$. Then automatically 
$c=c_p:= \frac{1+z_p}{1-z_p}$, and $a_p=a(z_p)$;
otherwise, the process may stop when $a_p=a(z_p)$.
Then again $c=c_p=\frac{1+z_p}{1-z_p}$.

One more proof of Theorem \ref{touching} can be obtained from the following fact which we establish for all $p>2$:  
 
\begin{lemma}
\label{separ}
The common tangent line $\ell_p$ separates the graphs of $L_{\alpha, a_p}$ and $h_{c_p}$ on $(z_p,1)$ if 
$p=\alpha(\alpha+1)>2$.
\end{lemma}
 
\begin{proof}
It is easy to calculate the slope of $\ell_p$:
$$
h_{c_p}'(z_p) = \frac{p}2 \Bigl(\frac{1+z_p}{2}\Bigr)^{p-1} \Bigl(1 + \frac{1+z_p}{1-z_p}\Bigr)=
p\Bigl(\frac{1+z_p}{2}\Bigr)^{p-1}\frac{1}{1-z_p}\,.
$$
Then 
$$
\ell_p(1) = p\Bigl(\frac{1+z_p}{2}\Bigr)^{p-1}\,.
$$
Later on we prove that
\begin{equation}
\label{ellp}
p\Bigl(\frac{1+z_p}{2}\Bigr)^{p-1}\ge 1\,,\qquad p\ge 2\,.
\end{equation}

In Lemma~\ref{convf1} below we prove that $L_\alpha$ is convex on $[z_p,1]$; hence, $L_{\alpha, a_p}$ lies above $\ell_p$. 
The function $h_{c_p}$ is concave on $(z_p, i_p)$ (recall that $i_p$ denote the unique inflection point of $h_{c_p}$). 
So $h_{c_p}<\ell_p$ on $(z_p, i_p]$. In particular, $h_{c_p}(i_p)< \ell_p(i_p)$. On $[i_p, 1]$ the function $h_{c_p}$ is convex. 
We have seen that at the left end of this interval it is lower than the line $\ell_p$. Inequality \eqref{ellp} shows that the 
same happens at the right end point $1$, because $h_{c_p}(1)=1$. This completes the proof of our lemma modulo  \eqref{ellp}.
   
To verify \eqref{ellp} we first establish
 
\begin{lemma}
\label{zpestimate}
$ \dfrac{1+z_p}{2} \ge \dfrac{p}{p+2}$, $p\ge 2$.
\end{lemma}
 
\begin{proof}
We have
$$
pL_\alpha(x) = -((1-x^2) L_\alpha'(x))'\,.
$$
We integrate this equality from $z_p$ to $1$:
$$
(1-z_p^2) L_\alpha'(z_p) = p\int_{z_p}^1 L_\alpha(x)\,dx\,.
$$
Since $L_\alpha$ is convex on $[z_p,1]$, the integrand 
is bigger than $L_\alpha'(z_p)(x-z_p)$. Therefore,
$$
(1-z_p^2) L_\alpha'(z_p)\ge \frac p2\, L_\alpha'(z_p) (1-z_p)^2\,.
$$
This is equivalent to
$$
\frac{1+z_p}{1-z_p} \ge \frac p2.
$$
which proves our assertion.
\end{proof}
 
To prove \eqref{ellp} and to finish the proof of Lemma \ref{separ} it remains to mention that 
$$
\Bigl(\frac{p}{p+2}\Bigr)^{p-1} \ge \frac1{p}\,,\qquad p\ge 2, 
$$
or $p^p\ge (p+2)^{p-1}$.  This elementary inequality is true for $p\ge 2$ with equality only for $p=2$.
\end{proof}
 
Thus, our best meeting point is $z_p$, and our best $c_p$ is
$$
\frac{1+z_p}{1-z_p}\,.
$$
As $x(p,a)=z_p$ and $c_p$ are already defined, we determine $a=a_p$ from  equalities above:
\begin{equation*}
a_p= \frac{p}2\biggl(\Bigl(\frac{1+z_p}{2}\Bigr)^{p-1} +\Bigl(\frac{1+z_p}{1-z_p}\Bigr)^p
\Bigl(\frac{1-z_p}{2}\Bigr)^{p-1}\biggr)\cdot L_{n}'(z_p)^{-1}\,.
\end{equation*}
 
\bigskip

\subsection{A candidate for the solution of \eqref{minus1}, \eqref{plus1}, \eqref{middle}}
We consider the function (recall that $p=n(n+1)$)
$$
g_p(s) :=\begin{cases} L_{n,a_p} (s), \,\, s\in [z_p, 1]\\ h_{c_p} (s), \,\, s\in [-1, z_p]\,.\end{cases}
$$
It is $C^1$-smooth, and it satisfies \eqref{minus1}: $Dg_p\le 0$. In fact, it actually satisfies
the equality $Dg_p=0$ on the interval $(z_p,1]$. On the interval $[-1, z_p]$, $g_p=h_{c_p}$ (recall that 
$h_{c_p}$ has an inflection point $i_p$ such that $c_p^p = \bigl(\frac{1+i_p}{1-i_p}\bigr)^{p-2}$ located 
to the right of $z_p$ satisfying the property $c_p^p = \bigl(\frac{1+z_p}{1-z_p}\bigr)^{p}$). Recall also that 
at the beginning of this section we checked that the inflection point of $h_c$ coincides with the point 
where $Dh_c$ changes the sign from negative to positive. Therefore, to the left of $i_p$  we have 
$Dh_{c_p}\le 0$, which implies that  to the left of $z_p$ we have  $Dg_p= Dh_{c_p}\le 0$.
 
Now let us check \eqref{plus1}, \eqref{middle}. Recall that $\tilde{D}g=\frac{Dg}{1-s^2} +Kg$. 
Therefore, to check \eqref{plus1} (i.e. $\tilde{D} g_p(s) \le 0$), it is sufficient to prove that
$$
-Kg_p(s) = (1-s^2) g_p''(s) + (2p-2) s g_p'(s) -(p^2-p) g_p(s)\ge 0\,.
$$
The function $g_p$ coincides with $h_{c_p}$ on $[-1,z_p]$. Note that
$$
Kh_c(s)=0
$$
identically on $[-1,1]$ for any $c$. So we need only to check the inequality $-KL_n(s)\ge 0$ on $[z_p,1]$.
We are to verify that 
$$
(1-s^2) L_n''(s) + (2p-2) s L_n'(s) -(p^2-p) L_n(s)\ge 0\,,
$$
and we have
\begin{equation}
\label{Ln1}
DL_n(s) = (1-s^2) L_n''(s) -2 s L_n'(s) +p L_n(s)= 0\,.
\end{equation}
Therefore, it suffices to check that
$$
p\bigl(2s L_n'(s) - pL_n(s)\bigr)\ge 0\,,\qquad s\in[z_p,1]\,.
$$
Using \eqref{Ln1} once more, we see that this follows from \eqref{convzp}.

\begin{lemma}
\label{middlelemma}
For $p=n(n+1)$, inequality \eqref{middle} is satisfied for the function $\phi(x,y) = (x+y)^p g_p(\frac{y-x}{x+y})$ 
such that $g_p= L_{n, a_p}$ to the right of $z_p$ and $g_p=h_{c_p}$ to the left of $z_p$, where 
$c_p=\frac{1+z_p}{1-z_p}$.
\end{lemma}
 
\begin{proof} If $s=\frac{y-x}{x+y}$ belongs to $(z_p,1]$, then $g_p=L_{n,a_p}$, and \eqref{middle} follows, see \eqref{BDA}.
To check \eqref{middle} for others $s$ is easy. Indeed, here $g_p= h_{c_p}$. For 
$\psi(x,y):= (x+y)^p ((y/(x+y))^p-c_p^p(x/(x+y))^p)= y^p-c_p^p x^p$ we have
$\psi_{xy}=0$, $ \psi_x/x +\psi_{xx} = - c_p^p p^2 x^{p-2}$, $ \psi_y/y +\psi_{yy} =  p^2 y^{p-2}$. 
Inequality \eqref{middle} for $\psi$ can be written as
$$
\bigl|(\psi_x/x +\psi_{xx})(\psi_y/y +\psi_{yy})\bigr|^{1/2} \le |\psi_x/x +\psi_{xx}|\,,
$$
or, equivalently,
\begin{equation}
y^{p-2} \le c_p^p x^{p-2}\,.
\label{x04}
\end{equation}
If $s<z_p$, then
$$
(1+s)^{p-2}\le c_p^p(1-s)^{p-2},
$$
and \eqref{x04} follows by \eqref{x03}.
\end{proof}

\begin{remark} \label{sg6}
We have proved \eqref{middle} for $p=n(n+1)$. However, our argument extends to all $p>2$.
\end{remark}

All the inequalities \eqref{minus}, \eqref{plus}, \eqref{middle} are now proved for $p=n(n+1)$. This shows that the constant in the orthogonal martingale estimate for such $p$ satisfies the inequality
\begin{equation*}
c_p\le \frac{1+z_p}{1-z_p}\,.
\end{equation*}

\section{Legendre equation. General facts}
\label{general}

Here we list some useful facts from \cite{C}.

Consider the equation 
$$
P(x) y'' +Q(x)y' +R(x) y=0
$$
near the point $x=x_0$, where $P, Q, R$ are analytic functions and
$$
(x-x_0)\frac{Q(x)}{P(x)}\,,\,\, (x-x_0)^2\frac{R(x)}{P(x)}
$$
are analytic functions in a neighborhood of $x_0$ (the regular singular point case). Suppose that the 
{\it indicial} polynomial
$$
r(r-1) + a_0 r +b_0 \,,
$$
where
$$
a_0 =\lim_{x\rightarrow x_0}(x-x_0)\frac{Q(x)}{P(x)}\,,\,\, b_0=\lim_{x\rightarrow x_0}(x-x_0)^2\frac{R(x)}{P(x)}
$$
has a double root $r=r_1$. Then our equation has two linearly independent solutions 
represented in a small half-neighborhood $(x_0,x_0+\varepsilon)$ by the formulas
\begin{align*}
f_1(x) &= (x-x_0)^{r_1} \sum_{n=0}^{\infty} a_n (x-x_0)^n\,,\\
f_2(x) &= f_1(x) \log\frac1{x-x_0} + (x-x_0)^{r_1}\sum_{n=1}^{\infty} b_n (x-x_0)^n\,;
\end{align*}
the series converge absolutely for $x\in(x_0-\varepsilon,x_0+\varepsilon)$, $a_0\not=0$.
 
In the case of the Legendre equation,
\begin{equation}
\label{Leg}
(1-x^2) y'' -2x y' +py=0,
\end{equation}
we will use these notations for a bounded and an unbounded solutions near $x_0=1$:
\begin{multline*}
f_1(x) = 1+ \frac{p}{2\cdot 1^2} (x-1) - \frac{p(1\cdot 2 -p)}{(2\cdot 1^2)(2\cdot 2^2)}(x-1)^2+\dots\\
+(-1)^{n+1}\frac{p(1\cdot 2 -p)\dots (n(n-1) -p)}{2^n (n!)^2} (x-1)^n +\dots\,.
\end{multline*}
For integer $\alpha$ this is a Legendre polynomial.
Furthermore, in a half-neighbor\-hood $(1-\varepsilon,1)$ we have 
\begin{equation*}
f_2(x) =f_1(x)\log\frac1{1-x} + H(x)\,,
\end{equation*}
where $H$ is real analytic in a neighborhood of $1$; $f_1$ and $f_2$ are real analytic on $(-1,1)$.

\section{The case $p>2$}
\label{pgr2}

To extend our solution of the main problem from the case $p=n(n+1)$ to the general case $p>2$, 
we need to prove a couple of lemmas. Denote by $z_p$ the rightmost zero of $f_1$ on the interval $[-1,1]$. 
We are going to prove two things: (1) for every solution of Legendre equation \eqref{Leg}, its rightmost zero 
on the interval $[-1,1]$ is at least $z_p$; (2) the solution $f_1$ is strictly convex on $[z_p, 1]$.
 
We need also Remark~\ref{sg6} to complete the reasoning for all $p>2$.
 
\begin{lemma}
\label{minzero}
For every solution of Legendre equation \eqref{Leg}, its rightmost zero on the interval $[-1,1]$ 
is at least $z_p$.
\end{lemma}
 
\begin{proof}
Note that $f_2(1)=+\infty$. Consider the Wronskian $W(x)= f_2'(x) f_1(x) -f_1'(x) f_2(x)$. 
Section~\ref{general} gives us that $W(x)\asymp \frac{1}{1-x}$, $x<1$. So $W$ is positive 
near $x=1$. We know that
$$
W'(x) = \frac{2x}{1-x^2} W(x)\,,
$$
and, hence, $W$ preserves the sign. Consider
$$
W(z_p)= f_2'(z_p) f_1(z_p) -f_1'(z_p) f_2(z_p)=-f_1'(z_p) f_2(z_p)\,.
$$
The function $f_1$ is positive on $[z_p,1]$, and it changes sign at $z_p$, so $f_1'(z_p)\ge 0$. 
If $f_1'(z_p)=0$, then by the Legendre equation, $f_1''(z_p)=0$, and differentiating the Legendre equation, 
we get $f_1^{(n)}(z_p)=0$ for all $n$. This is impossible as the analytic function $f_1$ would then vanish identically. Hence, $f_1'(z_p) >0$. We conclude that
$$
f_2(z_p)<0\,.
$$
Consider a linear combination $f_3= c_1 f_1 + c_2 f_2$ with a positive $c_2$. Then $f_3(z_p)<0$ by 
what we have just proved. Since $f_3(1)=+\infty$, $f_3$ must have a zero on $(z_p,1)$. For negative 
$c_2$ we have the same conclusion.
\end{proof}
 
\begin{lemma}
\label{convf1}
The function $f_1$ is strictly convex on $[z_p,1)$ for $p>2$.
\end{lemma}
 
\begin{proof}
We have
\begin{gather*} 
(1-x^2) f_1'' -2x f_1' +pf_1=0\,,\\
(1-x^2) f_1''' -4x f_1'' +(p-2)f_1'=0\,.
\end{gather*} 
Then
$$
f_1'(1) =\frac{p}2\,,\,\, f_1''(1) =\frac{p(p-2)}{8}\,.
$$
We have already observed that $f_1'(z_p)>0$. By the first equation above we get $f_1''(z_p)>0$ .
 
Let $x_1\in (z_p,1)$ be the first point where $f_1''(x_1) =0$. Then obviously $f_1'(x_1)>0$, and by 
the second equation above $f_1'''(x_1)<0$. So $f_1$ does change convexity to concavity passing 
through $x_1$. Furthermore, we have just seen that $f_1''(1) =\frac{p(p-2)}{8}>0$. So $f_1$ should 
change from concavity to convexity again, say at $x_3 \in (x_1, 1)$. Let $x_3$ be the closest to 
$x_1$ point with this property, so that $f_1''<0$ in between. Then there exists $x_2\in (x_1, x_3)$ 
such that $f_1'''(x_2) =0$, $ f_1''(x_2) <0$. Plug this to the second equation at the beginning 
of the proof and note that then
$$
f_1'(x_2) <0\,.
$$
Therefore, using again the Legendre equation we see that $f_1', f_1''$ will stay negative till the point $1$. This is impossible because they are  strictly positive at $1$.
\end{proof}
 
Let us observe that Remark~\ref{sg6} and Lemmas~\ref{minzero}, \ref{convf1} are the only ingredients we need to carry through the reasoning for general $p>2$. We just repeat the reasoning we used for the case $p=n(n+1)$ replacing the Legendre polynomials $L_n$ by $f_1$. We finally get the proof that for $p>2$ the best constant for the martingale transform of orthogonal martingales satisfies the inequality
\begin{equation*}
c_p\le \frac{1+z_p}{1-z_p}\,,
\end{equation*}
where $z_p$ is the largest zero of the solution $f_1$ of Legendre equation \eqref{Leg}: $(1-x^2)y'' -2x y'+py=0$ 
on the interval $[-1,1]$. We also proved that $z_p=\min_y\max \{z(y)\}$, where $z(y)$ denotes any zero $z$ 
of any nontrivial solution $y$ of \eqref{Leg} on the interval $[-1,1]$.

\section{Sharpness}
\label{sh}
 
We use the spherical coordinates 
$$ 
x=r\sin(\theta)\cos(\phi), y=r\sin(\theta)\sin(\phi), z=r\cos(\theta).
$$
The $\R^3$ Laplacian in spherical coordinates is 
$$
\Delta\psi=\frac{\partial_r(r^2\partial_r\psi)}{r^2}+
\frac{\partial_\theta(\sin(\theta)\partial_\theta\psi)}{r^2\sin\theta}+\frac{\partial^2_\phi\psi}{r^2\sin^2\theta}.
$$
Changing the variable $s=\cos\theta$, we define the obstacle function
$$ 
v(\theta)=\Bigl(\frac{1+\cos\theta}{2}\Bigr)^{\alpha(\alpha+1)}
-c^{\alpha(\alpha+1)}\Bigl(\frac{1-\cos\theta}{2}\Bigr)^{\alpha(\alpha+1)}.
$$
We associate to $v$ an auxiliary obstacle function in $\R^3$,
$$
V(r,\theta,\phi)=r^\alpha v(\theta)\,.
$$
Observe that the function $V$ has separated variables and azimuthal symmetry (i.e. no dependence on $\phi$). 

Consider now the {\it minimal} superharmonic function $\psi$, $\psi\ge V$. 
Then $\psi$ 
has the same symmetries as $V$ does. This is a consequence of taking infimum over 
the superharmonic majorants of the form $\psi(ax+by,-bx+ay,z)$ and 
$\frac{1}{\lambda^\alpha}\psi(\lambda x,\lambda y, \lambda z)$; this infimum which has both homogeneity 
and rotational invariance is again superharmonic. Thus the minimal superharmonic majorant 
$\psi(r,\theta,\phi)\geq V$ has the form
$$ 
\psi(r,\theta,\phi) = R(r)\Theta(\theta)\Phi(\phi) = r^\alpha \Theta(\theta).
$$
Applying the spherical Laplacian and using the azimuthal symmetry, we obtain (at least in the sense of distributions) that
$$
r^2\sin(\theta)\Delta\psi 
= \alpha(\alpha+1)r^{\alpha}\Theta(\theta)\sin\theta+r^{\alpha}(\Theta'(\theta)\sin\theta)'.
$$
Dividing by $r^\alpha\sin\theta$, we get
$$ 
\frac{\Delta\psi}{r^{\alpha-2}} = \frac{(\Theta'(\theta)\sin\theta)'}{\sin\theta}+\alpha(\alpha+1)\Theta(\theta).
$$
Here $\theta\in (0,\pi)$. The right hand side is the trigonometric form of the Legendre operator. The usual form 
$((1-s^2)y')'+\alpha(\alpha+1)y$ is obtained from this by the substitution $s=\cos \theta$ and 
$y(\cos\theta)=\Theta(\theta)$. In particular, we obtain that seeking for the minimal $y(s)\geq v(s)=h_c(s)$ such that 
$$ 
((1-s^2)y')'+\alpha(\alpha+1)y \le 0
$$ 
is equivalent to seeking for the minimal superharmonic $\psi\geq V$. It is well-known that  $\psi$ 
is harmonic wherever $\psi>V$. We know that the only harmonic function satisfying these homogeneity 
and symmetry conditions corresponds to a solution of the Legendre equation. 

We recall that the Bellman function $\B$ from the beginning of the paper generates a function $g$ 
on $[-1,1]$ satisfying \eqref{minus1}, \eqref{plus1}. We use only \eqref{minus1}, which shows that this particular $g$ is a supersolution of the Legendre equation: $(1-x^2) g'' -2x g' + p g\le 0$. Thus, the function 
$g$ generates a superharmonic majorant of $V$ given by the formula
$$
\Psi(r,\theta,\phi)=r^\alpha g(\cos(\theta))\,.
$$

Since $\psi\le\Psi$, there exists another solution  of \eqref{minus1} which actually satisfies the Legendre equation 
everywhere, where it is strictly bigger than $h_c$. Denote this solution of the Legendre equation by $L$. 
When $L$ and $h_c$ meet at a certain point $s_0$, it should be a point where they are tangent to each other. 
Otherwise, if $L>h_c$ on $(s_0, s_0+\epsilon)$, and $L(s_0)=h_c(s_0)$, then $L'(s_0)>h_c'(s_0)$. Define a 
new function: $f=L$ to the right of $s_0$ and $f=h_c$ to the left of $s_0$. It cannot be a solution of 
\eqref{minus1}. In fact, the inequality $L'(s_0)>h_c'(s_0)$ implies that $f''(s_0)$ is a positive delta function, 
and so satisfies at $s_0$ the inequality exactly opposite to \eqref{minus1}.

Thus, $L$  is inevitably tangent to $h_c$ at a meeting point. Furthermore, such pairs $(L, h_c)$, where 
$L$ is a solution of the Legendre equation and $h_c$ is an obstacle function, were already considered when 
we were looking for the smallest constant $c$. Thus, the minimal possible $c$ is $\frac{1+z_p}{1-z_p}$.

\section{Another way to prove sharpness}
\label{sh2}
 
The use of the Laplacian in $\R^3$ is of course very specific for the Legendre equation. The Legendre equation 
and spherical harmonics are close relatives. We want to show an approach that uses much less specifics of 
the ODE (still, it uses some of it).
 
Let
$$
P(x) y'' +Q(x) y' +R(x) y=0
$$
be an equation such that the functions $P, Q, R$ are real analytic, and let $x_0$ be 
a regular singular point (we choose $x_0=1$ as in the Legendre case, but this is of no importance). 
One has two solutions: a bounded one, $f_1$, and an unbounded one, $f_2$. In particular, this happens 
when the indicial equation has a double root $r_1=0$ as in the Legendre case. 
The reader may take a look at Section \ref{general}. We assume here that $f_1(1)>0$ and 
(unlike in Section \ref{general}) $f_2(1)=-\infty$. 
 
Let $b$ be the closest to $1$ zero of $f_1$, $b<1$. We deal only with the interval $[b,1]$, and assume that 
$P,Q,R$ keep sign on $[b,1)$.
 
\begin{lemma}
\label{supersol}
Let $g$ be a supersolution on $[b,1]$, that is $P(x) g'' +Q(x) g' +R(x) g\le 0$, and let $g(b)> 0$.
Then $g(1)=-\infty$.
\end{lemma}
 
Let us first explain why this lemma implies the sharpness of the constant $c_p$. Recall that the Bellman function $\B$ from the beginning of the paper generates a function $g$ on $[-1,1]$ satisfying 
\eqref{minus1}, \eqref{plus1}. We use only \eqref{minus1}, which shows that this particular $g$ is a supersolution of the  Legendre equation: $(1-x^2) g'' -2x g' + p g\le 0$. We constructed a special solution 
$L_{p, a_p}$ with the ``best" zero $z_p$, positive on $(z_p,1]$. The obstacle function $h_{c_p}$ lies below $L_{p, a_p}$ and touches it at $z_p$. If our $g$ were such that $g\ge h_c$ on $[-1,1]$, and $c<c_p$, 
then $g\ge h_c(z_p)>0$. Lemma \ref{supersol} shows that the inequality $g\ge h_c$ is impossible on 
$[z_p,1]$ because $h_c(1) =1$, $g(1)=-\infty$. Therefore, we must have $c\ge c_p$. And this is exactly the sharpness of our constant. 
 
It remains to prove our lemma.
 
\begin{proof}[Proof of Lemma~\ref{supersol}]
Consider two Wronskians:
\begin{align*}
W(x) &= f_2'(x) f_1(x)-f_1'(x) f_2(x),\\
\tilde{W}(x) &= g'(x) f_1(x)-f_1'(x) g(x)\,.
\end{align*}
 
Using that $g$ is a supersolution, and $f_1$ is a solution we can write on $(b,1)$:
$$
\tilde{W}' \le -\frac{Q}{P} \tilde{W}\,.
$$
On the other hand, we have everywhere 
$$
W' =-\frac{Q}{P} W\,.
$$
Combining these two relations we get
\begin{equation}
\label{W}
\tilde{W}' \le \frac{W'}{W} \tilde{W}\,.
\end{equation}
Furthermore, it is easy to see that $\lim_{x\to 1}W(x)=-\infty$ (because of the behavior of $f_2$). 
The Wronskian $W$ preserves the sign, so $W(x)<0$ on $[b,1)$. Therefore, \eqref{W} (after division by $W<0$) 
can be rewritten as
\begin{equation}
\label{W1}
(\tilde{W}/W)'\ge 0\,\,\text{on}\,\, [b,1)\,.
\end{equation}
 
Note that $\tilde{W}(b) = -g(b) f_1'(b) <0$, $W(b) =- f_2(b) f_1'(b)<0$. Set $\kappa=\frac{g(b)}{f_2(b)}>0$.
Inequality \eqref{W1} shows that 
$$
\tilde{W}(x)/W(x) \ge \kappa
$$
on $[b,1]$. Using the negativity of $W$, we get
\begin{equation}
\label{W2}
\tilde{W}(x)\le \kappa \,W(x)\,,\qquad x\in [b,1)\,.
\end{equation}
 
Note that $(g/f_1)'=\tilde{W}/f_1^2$, $(f_2/f_1)'=W/f_1^2$. Then \eqref{W2} becomes
\begin{equation*}
\Bigl(\frac{g}{f_1}\Bigr)'\le \kappa \,\Bigl(\frac{f_2}{f_1}\Bigr)'\,\,\text{on}\,\, [b,1)\,.
\end{equation*}
Hence,
\begin{equation*}
\frac{g}{f_1}\le \kappa \,\frac{f_2}{f_1} + \text{const}\,\,\text{on}\,\, [b,1)\,.
\end{equation*}
Now we see that $\lim_{x\to 1}g(x)=-\infty$. The lemma is proved.
\end{proof}

\section{Asymptotics of $z_p$}
\label{Asymptotics}

The Mehler-Heine formula (1868),
$$
\lim_{n\to\infty,\,n\in\mathbb N} L_n\bigl(\cos\frac xn\bigr)=J_0(x)
$$
establishes a relation between the Legendre polynomials $L_n$ and 
the Bessel function of zero order $J_0$. As a consequence 
(\cite[Theorem 8.1.2]{SZE}),
\begin{equation}
\lim_{n\to\infty,\,n\in\mathbb N}n(n+1)(1-z_{n(n+1)})=\frac{j^2_0}{2},
\label{rmk}
\end{equation}
where $j_0$ is the first positive zero of the Bessel function of zero order $J_0$.

Let $\beta>\alpha>1$, $D_\alpha f_\alpha=0$, $D_\beta f_\beta=0$, 
$f_\alpha(1)=f_\beta(1)=1$, and let $f_\alpha(x)>0$, 
$x_\alpha<x\le 1$, $f_\alpha(x_\alpha)=0$. If $f_\beta(x)> 0$ on  
$[x_\alpha,1]$, then $D_\alpha f_\beta\le 0$ on $[x_\alpha,1]$, and, hence, 
$f_\beta$ is a positive supersolution for $D_\alpha$ on $[x_\alpha,1]$ which is impossible by Lemma~\ref{supersol}. Thus, $z_p$ increases for 
$p\in(2,+\infty)$, and \eqref{rmk} gives us 
$$
\lim_{p\to\infty}p(1-z_{p})=\frac{j^2_0}{2}.
$$
 
\section{On Burkholder functions}
\label{Bfunction}
 
We did not find the formula for the function $\B$ from Section \ref{Bellf}. However, we have found the formula 
for the function $\Phi=\Phi_{c_p}$ from Section \ref{reduction}. Indeed, the function $\phi$ of two variables 
associated to $\Phi_{c_p}$ gives rise to a function $g_p$ of one variable on $[-1,1]$ such that 
the corresponding function $r^{\alpha} g_p(\cos \theta)$ (recall that $p=\alpha (\alpha +1)$) expressed 
in the spherical coordinates in $\R^3$ is the least superharmonic majorant of 
$$
V(r,\theta,\phi):=r^{\alpha}\biggl(\Bigl(\frac{1+\cos\theta}{2}\Bigr)^p - 
c_p^p\Bigl(\frac{1-\cos\theta}{2}\Bigr)^p\biggr)\,.
$$
How to recognize the least superharmonic majorant of a given function $V$? This is the function which is 
harmonic everywhere where it is strictly bigger than $V$. We constructed (using our $g_p$) exactly such a 
function. So we can completely restore $\Phi$ from $g_p$.

\section{One-sided orthogonality. Answers and questions.}
\label{ones}
 
Suppose that $Z$ is orthogonal, and $W$ is not. We always assume that $W$ is subordinate to $Z$. 
For $1<p\le 2$ we have cited an estimate
$$
\|W\|_p \le \sqrt{\frac{2}{p(p-1)}}\|Z\|_p\,.
$$
Is this sharp? We do not know. 
 
What if $p\ge 2$? Here is a result which we prove in \cite{BJV2}:
$$
\|W\|_p \le \sqrt{2}\frac{1-s^*_p}{s^*_p}\|Z\|_p\,,\qquad p\ge 2\,,
$$
where $s^*_p$ is the closest to $0$ zero of a bounded near $0$ solution of the Laguerre equation
\begin{equation}
s^2 \LLL'' +(1-s) \LLL' + p \LLL=0\,.
\label{lag5}
\end{equation}
This estimate is sharp, see \cite{BJV2}.

Suppose that $W$ is orthogonal, and $Z$ is not. For $p\ge 2$ we have cited an estimate
$$
\|W\|_p \le \sqrt{\frac{p(p-1)}{2}}\|Z\|_p\,.
$$
Is this sharp? We do not know.

What if $1<p\le 2$? Here is a result which we prove in \cite{BJV2}:
$$
\|W\|_p \le \frac1{\sqrt{2}}\frac{s_p}{1-s_p}\|Z\|_p\,,\qquad 1<p\le 2\,,
$$
where $s_p$ is the closest to $1$ zero of a bounded near $0$ solution of Laguerre equation
\eqref{lag5}. This estimate is sharp, see \cite{BJV2}.
 
 
  

\end{document}